\documentclass{article}
\usepackage{amsthm}
\usepackage{amsmath}
\usepackage{amssymb}
\usepackage{graphicx}
\newtheorem{theorem}{Theorem}[section]

\newtheorem{lemma}[theorem]{Lemma}
\theoremstyle{definition}

\newtheorem{conjecture}[theorem]{Conjecture}
\newtheorem{counterexample}[theorem]{Counterexample}
\usepackage{arxiv}
\usepackage{xcolor}
\usepackage[utf8]{inputenc} % allow utf-8 input
\usepackage[T1]{fontenc}    % use 8-bit T1 fonts
\usepackage{hyperref}       % hyperlinks
\usepackage{url}            % simple URL typesetting
\usepackage{booktabs}       % professional-quality tables
\usepackage{amsfonts}       % blackboard math symbols
\usepackage{nicefrac}       % compact symbols for 1/2, etc.
\usepackage{microtype}      % microtypography
\usepackage{enumitem}
\usepackage{comment}
\usepackage{wrapfig}
\usepackage{mwe}
\numberwithin{equation}{section}
\DeclareMathOperator{\Tr}{Tr}
\newcommand{\R}{\mathbb{R}}
\newcommand{\C}{\mathbb{C}}
\newcommand{\mba}{\mathbf{a}}
\newcommand{\mbb}{\mathbf{b}}
\newcommand{\Diag}{\text{Diag}}
\newcommand{\mbx}{\mathbf{x}}
\newcommand{\mby}{\mathbf{y}}

\newcommand{\mbv}{\mathbf{v}}

\newcommand{\diff}{\text{ d}}

\newcommand{\s}{\textcolor{white}{........}}
\title{Matrix Rearrangement Inequalities Revisited}

\author{
\textbf{Victoria M~Chayes}\\
Department of Mathematics\\
Rutgers University\\
Piscataway, NJ 08854 \\
\texttt{vc362@math.rutgers.edu}
}

\begin{document}

\maketitle

\begin{abstract}
Let $||X||_p=\Tr[(X^\ast X)^{p/2}]^{1/p}$ denote the $p$-Schatten norm of a matrix $X\in M_{n\times n}(\C)$, and $\sigma(X)$ the singular values with $\uparrow$ $\downarrow$ indicating its increasing or decreasing rearrangements. We wish to examine inequalities between $||A+B||_p^p+||A-B||_p^p$, $||\sigma_\downarrow(A)+\sigma_\downarrow(B)||_p^p+||\sigma_\downarrow(A)-\sigma_\downarrow(B)||_p^p$, and $||\sigma_\uparrow(A)+\sigma_\downarrow(B)||_p^p+||\sigma_\uparrow(A)-\sigma_\downarrow(B)||_p^p$ for various values of $1\leq p<\infty$. It was conjectured in \cite{Carlen2006} that a universal inequality $||\sigma_\downarrow(A)+\sigma_\downarrow(B)||_p^p+||\sigma_\downarrow(A)-\sigma_\downarrow(B)||_p^p\leq ||A+B||_p^p+||A-B||_p^p \leq ||\sigma_\uparrow(A)+\sigma_\downarrow(B)||_p^p+||\sigma_\uparrow(A)-\sigma_\downarrow(B)||_p^p$ might hold for $1\leq p\leq 2$ and reverse at $p\geq 2$, potentially providing a stronger inequality to the generalization of Hanner's Inequality to complex matrices $||A+B||_p^p+||A-B||_p^p\geq (||A||_p+||B||_p)^p+|||A||_p-||B||_p|^p$. We extend some of the cases in which the inequalities of \cite{Carlen2006} hold, but offer counterexamples to any general rearrangement inequality holding. We simplify the original proofs of \cite{Carlen2006} with the technique of majorization. This also allows us to characterize the equality cases of all of the inequalities considered. We also address the commuting, unitary, and $\{A,B\}=0$ cases directly, and expand on the role of the anticommutator. In doing so, we extend Hanner's Inequality for self-adjoint matrices to the $\{A,B\}=0$ case for all ranges of $p$.
\end{abstract}

\section{Introduction}
\label{intro}

\s It has been of great interest to extend Hanner's Inequality for $L^p$ spaces \begin{equation}
||f+g||_p^p+||f-g||_p^p\geq (||f||_p+||g||_p)^p+|||f||_p-||g||_p|^p
\end{equation}
for $1\leq p\leq 2$ to the non-communative analogue in $C^p$ \begin{equation}\label{h2}
||A+B||_p^p+||A-B||_p^p\geq (||A||_p+||B||_p)^p+|||A||_p-||B||_p|^p.
\end{equation}

\s In \cite{Carlen2006}, Carlen and Lieb proposed the following two conjectures for their potential pertinence to proving \ref{h2}: 
\begin{conjecture}\label{c1}
For all $1 \leq p \leq 2$, and all complex-valued $n \times n$ matrices $A$ and $B$, \begin{equation}
||A+B||_p^p+||A-B||_p^p\geq ||\sigma_\downarrow(A)+\sigma_\downarrow(B)||_p^p+||\sigma_\downarrow(A)-\sigma_\downarrow(B)||_p^p.
\end{equation}
For $p > 2$, the inequality reverses.
\end{conjecture}
\begin{conjecture}\label{c2}
For all $1 \leq p \leq 2$, and all complex-valued $n \times n$ matrices $A$ and $B$, \begin{equation}
||A+B||_p^p+||A-B||_p^p\leq ||\sigma_\uparrow(A)+\sigma_\downarrow(B)||_p^p+||\sigma_\uparrow(A)-\sigma_\downarrow(B)||_p^p.
\end{equation}
For $p > 2$, the inequality reverses.
\end{conjecture}

\s For these, the authors proved Conjecture \ref{c1} in the case $A\geq B\geq 0$ and $1\leq p\leq 2$; and proved Conjecture \ref{c2} in the case $A\geq |B|\geq 0$ and $1\leq p\leq 2$. We note a missing requirement in \cite{Carlen2006} used in the proof for Conjecture \ref{c2} in those conditions is also that $\sigma_n(A)\geq \sigma_1(B)$. Lemma 2.1 in \cite{Tomczak1974} proves that Conjecture \ref{c1} holds for all matrices and $p=2k$, $k\in \mathbb{N}$. To the best of our knowledge, no further work has been done on the subject.

\s If Conjecture \ref{c1} were true in general, with an additional application of Hanner's Inequality on the sequences of singular values, the non-commutative Hanner's Inequality for matrices would be proven in general; currently, it is only known for $A+B, A-B\geq 0$ for all $p$, or general $A$ and $B$ in the ranges $1\leq p\leq \frac{4}{3}$ and $p\geq 4$ \cite{Ball1994}. 

\s In this paper we extend the range of Conjecture \ref{c1} with the requirements $A\geq B\geq 0$ to $2\leq p\leq 3$, and we prove Conjecture \ref{c2} in the $A+B, A-B\geq 0$, $\sigma_n(A)\geq\sigma_1(B)$ case for the range $1\leq p\leq 3$. We prove both conjectures for the full range of $p$ in the commuting case. We prove Conjecture \ref{c1} in the case that $A$ and $B$ are both unitary, and in the case when $A$ and $B$ are self-adjoint and $\{A,B\}=0$. However, we demonstrate that both conjectures are false in general. Section 2 gives a background to majorization, which is the primary technique that we use in our proofs. Section 3 presents the extensions of the conjectures' requirements and ranges, and general counterexamples.

\s The key observation as to why the conjectures cannot hold in general is that if the matrix $B$ is taken to be unitary, all its singular values are equal to 1, and therefore there is no distinction between the ``aligned" and the ``up-down" rearrangements. If both conjectures were true, this would imply equality everywhere when $B$ is unitary, which can easily be numerically confirmed as false. 

\s The fact that these re-arrangement inequalities do not hold in general is notable, because the analogue to Conjecture \ref{c1} with complex functions and the spherically symmetric decreasing rearrangement is shown in \cite{Carlen2006} to hold. Therefore, we see directly that the non-commutativity of the matrices ruins a commutative identity. In disproving Conjecture \ref{c1} in general, we also rule it out as a method to attempt to extend Hanner's inequality to $C^p$. 

\s We will use the following notation throughout this paper: $\sigma(X)$ denotes the vector of singular values of a matrix $X$, assumed to be in descending order unless $\sigma_\uparrow(X)$ is specified; $\sigma_\downarrow(X)$ may then be used for emphasis. The norm $||\cdot||_p$ may either indicate the vector $p$-norm or the $p$-Schatten norm dependent on context. We use for a vector $\mbv$ the notation $[\mbv]$ to indicate the matrix $[\Diag(\mbv)]$. 

\section{Majorization}
\label{majorization}

\s Let $\mba,\mbb\in\R^n$ with components labeled in descending order $a_1\geq\dots\geq a_n$ and $b_1\geq\dots\geq b_n$. Then $\mathbf{b}$ weakly majorizes $\mathbf{a}$, written $\mathbf{a}\prec_{w} \bf{b}$, when \begin{equation}
\sum_{i=1}^ka_i\leq \sum_{i=1}^k b_i, \qquad 1\leq k \leq n
\end{equation}
and majorizes $\mba\prec\mbb$ when the final inequality is an equality. Weak log majorization $\mba\prec_{w(\log)}\mbb$ is similarly defined for non-negative vectors as \begin{equation}
\prod_{i=1}^ka_i\leq \prod_{i=1}^k b_i, \qquad 1\leq k \leq n
\end{equation}
with log majorization $\mba\prec_{(\log)}\mbb$ when the final inequality is an equality. An important fact is that log majorization and weak log majorization both imply weak majorization \cite{ANDO1989} [Lemma 1.8].

\s Note that it is not necessary that the vectors $\mba$ and $\mbb$ be in descending order---majorization is explicitly defined with respect the the rearrangements of the values in descending order. We define all of the above majorization for matrices, ie $A\prec B$ and all variations, when the singular values considered as a vector are majorized $\sigma(A)\prec\sigma(B)$. All operators stated for majorization (ie $f(\mba)$ or $\mba\mbb$) should be considered to be applied entrywise to the vectors (ie $(f(a_1),\dots,f(a_n))$ or $(a_1b_1,\dots,a_nb_n)$.

\s Majorization holds the following vital property:

\begin{theorem}{(Hardy, Littlewood, and P\'{o}lya \cite{Hardy1929} \cite{Hardy2}; Tomi\'{c}, Weyl \cite{tomic1949} \cite{Weyl1949} )}\label{maj} 
Suppose $\mba\prec_{w}\mbb$. Then for any function $f:\R\rightarrow\R$ that is increasing and convex on the domain containing all elements of $\mba$ and $\mbb$, \begin{equation}
\sum_{i=1}^n f(a_i)\leq\sum_{i=1}^n f(b_i).
\end{equation}
If $\mba\prec\mbb$, the `increasing' requirement can be dropped.
\end{theorem}

\s An immediate yet highly useful lemma follows:

\begin{lemma}\label{weakpower}
Let $\mba,\mbb\in\R_n^+$. Suppose $\mba\prec_{w}\mbb$. Then $\mba^s\prec_{w}\mbb^s$ for all $s\geq 1$.
\end{lemma}

\s Log majorization also allows us to characterize equality cases:

\begin{lemma}(Hiai \cite{GTEqualityCases} [Lemma 2.2])\label{strict}
Let $\Phi:\R^+\rightarrow\R^+$ be a strictly convex increasing function. Then $\mba\prec_{(\log)}\mbb$ and $\sum_{i=1}^n\Phi(a_i)=\sum_{i=1}^n\Phi(b_i)$ implies $\mba=\Theta\mbb$ for some permutation matrix $\Theta$.
\end{lemma}
As exponentiating is strictly convex, an immediate corollary is $\mba\prec_{(\log)}\mbb$ and $\mba\prec\mbb$ implies $\mba=\Theta\mbb$.

\s Majorization is an incredibly powerful technique in matrix analysis used to prove numerous inequalities about eigenvalues and singular values of matrices, powers of products of positive semidefinite matrices, Golden-Thompson-like inequalities, and more. A good overview of the techniques and important results can be found in \cite{Hiai2014}, \cite{MajBook}. The two results that we will need for this paper will regard the eigenvalues of the sums of Hermitian matrices:

\begin{theorem}{(Fan \cite{Fan})}\label{abmaj} 
	Let $A,B\in M_{n\times n}(\C)$ be self-adjoint. Then \begin{equation}
		\lambda(A+B)\prec\lambda(A)+\lambda(B).
	\end{equation}
\end{theorem}

and the singular values of products of general matrices:

\begin{theorem}{(Gel'fand and Naimark \cite{GelNai50})}\label{sabweak} 
	Let $A,B\in M_{n\times n}(\C)$. Then \begin{equation}
		\sigma(AB)\prec_{(\log)}\sigma(A)\sigma(B).
	\end{equation}
\end{theorem}

\s We will also need a fairly intuitive lemma that to our knowledge has not yet been addressed in existing literature, characterizing the concatenation of majorized vectors:

\begin{lemma}\label{weaksum}
Let $\mbx\prec_w\mby$, and $\mba\prec_{w}\mbb$ be non-negative vectors labeled in descending order.  Then $\mbx\mba\prec_w\mby\mbb$.
\end{lemma}
\begin{proof}
We can write the components of $\mby$ as $y_{n-1}=y_n+\epsilon_1$, $y_{n-2}=y_n+\epsilon_1+\epsilon_2$, $\dots$, $y_1=y_n+\epsilon_1+\dots+\epsilon_{n-1}$ where $\epsilon_i\geq 0$. Then applying $\mba\prec_{w}\mbb$ \begin{align}
\epsilon_{n-1}a_1&\leq \epsilon_{n-1}b_1 \\	
\epsilon_{n-2}\left(\sum_{i=1}^2a_i \right)&\leq \epsilon_{n-2}\left(\sum_{i=1}^2b_i \right)\\
&\vdots \\
y_n\left(\sum_{i=1}^na_i \right)&\leq y_n\left(\sum_{i=1}^nb_i \right)
\end{align}
and summing them all together, \begin{equation}
\sum_{i=1}^n y_ia_i\leq \sum_{i=1}^n y_ib_i.	
\end{equation}

\s Applying the same splitting argument to $a_i$ with $\mbx\prec_{w}\mby$ gives \begin{equation}
\sum_{i=1}^n x_ia_i\leq \sum_{i=1}^n y_ia_i,
\end{equation}	
and stringing the two inequalities together \begin{equation}
\sum_{i=1}^n x_ia_i\leq\sum_{i=1}^n y_ib_i.	
\end{equation}

\s Finally, nothing that when $\mba\prec_w\mbb$, the first $k$th components maintain the weak majorization relationship $(a_1,\dots,a_k)\prec_w(b_1,\dots,b_k)$, applying the argument to the $k$th components gives the desired result.
\end{proof}
\s Note that the above technique can be expressed compactly as the weigted sum of Ky-Fan norms for matrices $[\mba], [\mbb], [\mbx]$, and $[\mby]$, and leveraging the matrix majorization result that $A\prec_w B$ implies $|||A|||\leq |||B|||$ for every unitarily invariant norm $|||\cdot |||$. 

\section{Extensions And Counterexamples}
\label{maintheorem}

\s First, we address rearrangement of commuting matrices:

\begin{theorem}\label{ct}
Let $A,B\in M_{n\times n}(\C)$ be two self-adjoint matrices that commute. Then \begin{equation}\label{commuteineq}
||\sigma(A)+\sigma(B)||_p^p+||\sigma(A)-\sigma(B)||_p^p\leq ||A+B||_p^p+||A-B||_p^p\leq ||\sigma_\uparrow(A)+\sigma_{\downarrow}(B)||_p^p+||\sigma_\uparrow(A)-\sigma_{\downarrow}(B)||_p^p	
\end{equation}
for $1\leq p\leq 2$, and the inequality reverses for $p\geq 2$. Furthermore, there is equality in either inequality for $p\neq 1,2$ if and only if $A$ and $B$ are aligned in the extremized arrangement. 
\end{theorem}
\begin{proof}
In the mutually diagonalizable basis, we can write \begin{equation}
A=\begin{bmatrix} \lambda_1(A) & & \\ & \ddots & \\ & & \lambda_n(A)\end{bmatrix},\qquad B=\begin{bmatrix} \lambda_{i_1}(B) & & \\ & \ddots & \\ & & \lambda_{i_n}(B)\end{bmatrix}.
\end{equation}
Then we note that the singular values of $A+B$ and $A-B$ can be grouped as \begin{equation}\label{samerow}
\{|\lambda_j(A)\pm \lambda_{i_j}(B)|  \}=\{||\lambda_j(A)|\pm |\lambda_{i_j}(B)||  \}=\{|\sigma_{i}(A)\pm \sigma_{k_i}(B)|  \}.
\end{equation}
Re-labeling to preserve the pairings above, we consider the functions \begin{align}
f(x)&=\sum_{i=1}^n \sigma_i(A)\chi_{[i-1,i)}(x)\label{r1}	\\
g(x)&=\sum_{i=1}^n \sigma_{k_i}(B)\chi_{[i-1,i)}(x)	\label{r2}
\end{align}
and let $f^\ast$ and $g^\ast$ denote the spherically symmetric decreasing rearrangements. 

\s We will need an extension of the Riesz rearrangement inequality: 

\begin{lemma}{(Almgren, Lieb \cite{almlieb1989} Theorem 2.2). }\label{sdl}
Let $F:\R^+\times\R^+\rightarrow\R^+$ be a continuous function such that $F(0,0)=0$ and \begin{equation}\label{CP}
F(u_2, v_2)+F(u_1, v_1) \geq F(u_2, v_1)+F(u_1, v_2)\end{equation}
whenever $u_1\geq u_2\geq 0$ and $v_1\geq v_2\geq 0$. Then for $f,g$ in the class $\mathcal{C}_0$ (ie $f^\ast$, $g^\ast$ are well-defined), the inequality \begin{equation}\label{CP2}
\int F(f(x),g(x)) \diff\mathcal{L}^n x \leq \int F(f^\ast(x),g^\ast(x)) \diff\mathcal{L}^n x\end{equation}
holds. If condition \ref{CP} is reversed, the inequality \ref{CP2} is reversed.
\end{lemma}

\s The following technique is inspired by \cite{Carlen2006} [Lemma 1.1], which in fact proves a more general theorem on symmetric decreasing arrangements of general complex functions. For the left half of our inequality, we choose 
\begin{equation}
F(x,y)=|x+y|^p+|x-y|^p.
\end{equation}

\s We see that $\partial^2F(x,y)/\partial x\partial y\leq 0$ when $1< p\leq 2$, with the inequality switching at $p=2$, satisfying the condition of Equation \ref{CP}. Then
\begin{equation}
||f+g||_p^p+||f-g||_p^p\geq ||f^\ast+g^\ast||_p^p+||f^\ast-g^\ast||_p^p	
\end{equation} 
for $1< p\leq 2$ (and taking the limit for $p=1$), with the inequality switching for $p\geq 2$. As $||f\pm g||_p^p=||A\pm B||_p^p$ and $||f^\ast\pm g^\ast||_p^p=||\sigma(A)\pm\sigma(B)||_p^p$, the left half of Equation \ref{commuteineq} is proven.

\s For the right half, without loss of generality, we can assume that $B$ is invertible; otherwise, we could consider a limit of perturbations. As the inequality for matrices $A,B$ holds if and only if it holds for $cA, cB$ for some scaling constant $c$, we can further assume that the largest singular value of $B$ is equal to 1. We define piecewise functions such as \begin{equation}
	F(x,y)=\begin{cases} \left|x+\frac{1}{y}\right|^p +\left|x-\frac{1}{y}\right|^p	& x\geq 0,\; y\geq 1 \\
		e^{1-\frac{1}{y}}\left( \left|x+\frac{1}{y}\right|^p +\left|x-\frac{1}{y}\right|^p \right)	& x\geq 0, \;0\leq y<1 \end{cases}
\end{equation}
for $1\leq p\leq 2$ and \begin{equation}
	F(x,y)=\begin{cases} \left|x+\frac{1}{y}\right|^p +\left|x-\frac{1}{y}\right|^p	& x\geq 0,\; y\geq 1 \\
		e^{p(1-y)}\left(|x+y|^p+|x-y|^p  \right)	& x\geq 0, \;0\leq y<1 \end{cases}
\end{equation}
for $p\geq 2$. The values of the function $F$ that we care about will be in the $y\geq 1$ range; these are merely examples of functions existing that satisfy the necessary conditions to apply Theorem \ref{sdl}.

\s It can be readily confirmed that $F(x,y)$ is continuous, and by exponential domination in the limit $F(0,0)=0$. We therefore calculate the partial derivative on each piece and see that $\partial^2F(x,y)/\partial x\partial y\geq 0$ when $1< p\leq 2$, with the inequality reversing at $p=2$, satisfying the condition of Equation (\ref{CP}). Then letting \begin{align}
	f(x)&=\sum_{i=1}^n \sigma_i(A)\chi_{[i-1,i)}(x)\label{r3}	\\
	g(x)&=\sum_{i=1}^n (\sigma_{k_i}(B))^{-1}\chi_{[i-1,i)}(x)	\label{r4}
\end{align}
and comparing $\int F(f(x),g(x)) \diff x$ and $\int F(f^\ast(x),g^\ast(x)) \diff x$ (and taking the limit for $p=1$), the full inequality is proven.

\s To characterize the equality cases, we can consider $F$ as the limit of twice-differentiable functions with $W(x-y)=\epsilon^{-1}\exp(|x-y|/\epsilon)$ and take the limit as $\epsilon\rightarrow 0$. One can express \begin{equation}\begin{split}
		\int \int F&(f(x),g(y))W(x-y)\diff\mathcal{L}^nx \diff\mathcal{L}^n y\\ &=\int\int F_{12}(s,t)\left[ \int\int \chi_{\{ f>s \}}(x)\chi_{\{ g>t \}}(y)W(x-y)\diff\mathcal{L}^nx \diff\mathcal{L}^n y\right]\diff\mathcal{L}^1s \diff\mathcal{L}^1 t
	\end{split}
\end{equation}
and apply the Riesz rearrangement inequality to the interior integral. Note that for piecewise $f$ and $g$ as defined in Equations (\ref{r1}) and (\ref{r2}) or (\ref{r3}) and (\ref{r4}), there is strict inequality in application of the Riesz rearrangement inequality if the functions are not aligned \cite{rreq}. When we take the limit, $W$ converges to the delta distribution, and we have \begin{equation}\begin{split}\label{i}
		\int\int &F_{12}(s,t)\left[ \int\int \chi_{\{ f>s \}}(x)\chi_{\{ g>t \}}(y)\delta(x-y)\diff\mathcal{L}^nx \diff\mathcal{L}^n y\right]\diff\mathcal{L}^1s \diff\mathcal{L}^1 t \\ & \geq \int\int F_{12}(s,t)\left[ \int\int \chi_{\{ f^\ast>s \}}(x)\chi_{\{ g^\ast>t \}}(y)\delta(x-y)\diff\mathcal{L}^nx \diff\mathcal{L}^n y\right]\diff\mathcal{L}^1s \diff\mathcal{L}^1 t 
	\end{split}
\end{equation}

\s Suppose, as in our case, that $F_{12}(s,t)>0$ for $s,t>0$, and that the matrices $A$ and $B$ are unaligned. Then there exists at least one interval $[i-1,i)$ and values $0<c_0<c_1$ such for $x\in [i-1,i)$ and $c_0<c<c_1$ where $f^\ast(x), g^\ast(x)>c$, but either $f(x)>c$ and $g(x)\leq c$ or $f(x)\leq c$ and $g(x)>c$; otherwise, the matrices are in fact aligned to the equality case. Then $ \chi_{\{ f>s \}}(x)\chi_{\{ g>t \}}(x)=0$ and $ \chi_{\{ f^\ast>s \}}(x)\chi_{\{ g^\ast>t \}}(x)=1$ for $s,t\in [c_0, c_1]$. This means that the interior integral will be strictly greater when $f$ and $g$ are in decreasing rearrangements, and we conclude that the inequality (\ref{i}) must be strict.  
\end{proof}

\s Next, we address the case when anticommutator $\{A,B \}=0$.

\begin{theorem}
Let $A, B\in M_{n\times n}(\C)$ be self-adjoint such that $AB+BA=0$. Then \begin{equation}\label{tt1}
||A+B||_p^p+||A-B||_p^p\geq ||\sigma(A)+\sigma(B)||_p^p+||\sigma(A)-\sigma(B)||_p^p
\end{equation}
for $1\leq p\leq 2$, with the inequality reversing for $p\geq 2$.
\end{theorem}
\begin{proof}
We note that as $\lambda(X^2)=\sigma(X^2)$ for sef-adjoint $X$, \begin{equation}
||A+B||_p^p=\sum_{i=1}^n\lambda_i((A+B)^2)^{p/2}=\sum_{i=1}^n\lambda_i(A^2+B^2)^{p/2}=\sum_{i=1}^n\lambda_i((A-B)^2)^{p/2}=||A-B||_p^p	
\end{equation}

\s When $1\leq p\leq 2$, we make use of the majorization identity of Theorem \ref{abmaj} of $\lambda(A+B)\prec\lambda(A)+\lambda(B)$ and the fact that $f(x)=x^{p/2}$ is concave to conclude that \begin{align}
||A+B||_p^p+||A-B||_p^p&=2\sum_{i=1}^n\lambda_i(A^2+B^2)^{p/2}	\\
&\geq 2 \sum_{i=1}^n(\lambda_i(A^2)+\lambda_i(B^2))^{p/2}\\
&= 2 \sum_{i=1}^n(\sigma_i(A)^2+\sigma_i(B)^2)^{p/2}\\
&= 2 \sum_{i=1}^n\left(\frac{(\sigma_i(A)+\sigma_i(B))^2}{2}+\frac{(\sigma_i(A)-\sigma_i(B))^2}{2}\right)^{p/2} \\
&\geq \sum_{i=1}^n((\sigma_i(A)+\sigma_i(B))^2)^{p/2}+((\sigma_i(A)-\sigma_i(B))^2)^{p/2} \\
&=||\sigma(A)+\sigma(B)||_p^p+||\sigma(A)-\sigma(B)||_p^p
\end{align}
An identical argument for $p\geq 2$ with reversed inequalities can be made now leveraging convexity of $x^{p/2}$.

\s Note that this proof extends to general $A,B$ when $AB^\ast+BA^\ast=0$.
\end{proof}

\s The unitary case gives some insight to the role of the anticommutator. 

\begin{theorem}\label{tu}
Let $U,V\in M^{n\times n}(\C)$ unitary. Then \begin{equation}\label{u1}
||U+V||_p^p+||U-V||_p^p\geq 2^pn
\end{equation}
for $1\leq p\leq 2$, with the inequality switching for $p\geq 2$.	There is equality for $p\neq 2$ if and only if $U=V$. The extremization of the inequality is directly dependent on $\sigma(UV+VU)$, with greatest difference when $\{U,V \}=0$.
\end{theorem}
\begin{proof}
Note that Equation \ref{u1} can be directly derived from \cite{mccarthy1967} the Clarkson type inequalities 
\begin{equation}
2(||A||_p^p+||B||_p^p)\leq ||A+B||_p^p+||A-B||_p^p \leq 2^{p-1}(||A||_p^p+||B||_p^p)
\end{equation}
for $p\geq 2$ and reversing for $1\leq p\leq 2$; and in fact can be seen from direct matrix inequalities of Theorems 2.1 and 2.5 of \cite{BOURIN2020170}. However, we can use majorization to examine this inequality on the level of the eigenvalues to see the direct role of the anticommutator. 

We can assume without loss of generality that $U$ and $V$ are also self-adjoint; otherwise, consider the unitary matrices \begin{equation}
	\widehat{U}=\begin{bmatrix} 0 & U \\ U^\ast & 0\end{bmatrix}, \qquad \widehat{V}=\begin{bmatrix} 0 & V \\ V^\ast & 0\end{bmatrix}
\end{equation}
then the inequality holds for $U, V$ if and only if it holds for $\widehat{U}, \widehat{V}$ by dividing by the appropriate factor of 2. 

\s Once more we will make use of $\lambda((U\pm V)^2)=\sigma((U\pm V)^2)$. We note that $UV+VU$ is a Hermitian matrix, and as $||UV+VU||\leq ||UV||+||VU||\leq 2||U||||V||=2$, the eigenvalues of $UV+VU$ must be within the interval $[-2,2]$, and can be written as $2\cos(\theta_j)$. Then \begin{align}
||U+V||_p^p+||U-V||_p^p&=||(U+V)^2||_{p/2}^{p/2}+||(U-V)^2||_{p/2}^{p/2} \\
&=||\lambda(2I+UV+VU)||_{p/2}^{p/2}+||\lambda(2I-UV-VU)||_{p/2}^{p/2}\\
&=||2+\lambda(UV+VU)||_{p/2}^{p/2}+||2-\lambda(UV+VU)||_{p/2}^{p/2}\\
&=\sum_{j=1}^n2^{p/2}|1+\cos(\theta_j)|^{p/2}+2^{p/2}|1-\cos(\theta_j)|^{p/2}
\end{align}

\s The function $f(\theta)=(1+\cos(\theta))^s+(1-\cos(\theta))^s$ can be examined on the interval $(0, \frac{\pi}{2})$. It has derivative $s\sin(\theta)[(1-\cos(\theta))^{s-1}-(1+\cos(\theta))^{s-1}]$, which can only be 0 at $\theta=0, \frac{\pi}{2}$. It is immediately confirmed that the function monotone for all $s\geq 0$ is minimized at $\theta=0$ and maximized at $\theta=\frac{\pi}{2}$ for $0\leq s\leq 1$, with the maximum and minimum reversing for $s\geq 1$.  Therefore, \begin{align}
\sum_{j=1}^n2^{p/2}|1+\cos(\theta_j)|^{p/2}+2^{p/2}|1-\cos(\theta_j)|^{p/2} \geq \sum_{j=1}^n2^{p/2}|2|^{p/2}=n2^p
\end{align}
and the rearrangement inequality holds as desired for $1\leq p\leq 2$, with the inequality reversing for $p\geq 2$. As the desired extrema are reached only at $\theta=0$, then if there is equality for $p>2$, we must have $\theta_j=0$ for all $j$, and hence $UV=VU=I$. As $U$ is self-adjoint and unitary, we know that $U^{-1}=U$, and hence we conclude $V=U$. The alternative extrema are reached when $\theta_j=\frac{\pi}{2}$ for all $j$, and hence $UV+VU=0$.
\end{proof}

\s We finally expand upon the ranges of Conjectures \ref{c1} and \ref{c2} as originally seen in \cite{Carlen2006}, and comment on how this can lead to counterexamples. 

\begin{theorem}\label{t1}
Let $A,B\in M_{n\times n}(\C)$ be self-adjoint with $A\geq B\geq 0$. Then \begin{equation}\label{e2}
||A+B||_p^p+||A-B||_p^p\geq ||\sigma(A)+\sigma(B)||_p^p+||\sigma(A)-\sigma(B)||_p^p		
\end{equation}
with the inequality reversing for $2\leq p\leq 3$. There is equality for $p\neq 1,2$ if and only if there is equality in the entire range $1\leq p\leq 3$.
\end{theorem}
\begin{proof}
For a positive matrix $C$ and $1<p<2$, for positive normalization constant $k_p$ we have \begin{equation}
C^p=k_p\int_0^\infty\left(\frac{C}{t^2}-\frac{1}{t}+\frac{1}{t+C} \right)t^p dt.
\end{equation}
We can therefore express the difference between sides in Equation \ref{e2} for $1<p<2$ by the integral representation after cancellation as \begin{equation}\label{e00}
k_p\Tr\left[\int_0^\infty\left(\frac{1}{A+B+t}+\frac{1}{A-B+t}-\frac{1}{\sigma(A)+\sigma(B)+t}-\frac{1}{\sigma(A)-\sigma(B)+t} \right)t^p dt\right].
\end{equation}
In \cite{Carlen2006} it is proven that when $A\geq B\geq 0$, this integrand is always positive semidefinite. Therefore, the integral is zero is and only if it is zero everywhere, if and only if Equation \ref{e00} is zero. This would happen independent of $p$, and hence if there is equality for some $1<p<2$, there must be equality for all $1\leq p\leq 2$.

\s To extend the range to $2\leq p\leq 3$, we see that \begin{align}
C^p=k_pC\int_0^\infty\left(\frac{C}{t^2}-\frac{1}{t}+\frac{1}{t+C} \right)t^p dt&=k_p\int_0^\infty\left(\frac{C^2}{t^2}-\frac{C}{t}+\frac{C}{t+C} \right)t^p dt\\
&=k_p\int_0^\infty\left(\frac{C^2}{t^3}-\frac{C}{t^2}+\frac{1}{t}-\frac{1}{t+C} \right)t^{p+1} dt
\end{align}
The first three terms of the integral cancel completely between each side of Equation \ref{e2}, and now as the sign of the final term is reversed, the argument for $1<p<2$ is reversed.
\end{proof}

\begin{wrapfigure}[26]{l}{0.5\textwidth}
	\includegraphics[width=0.45\textwidth]{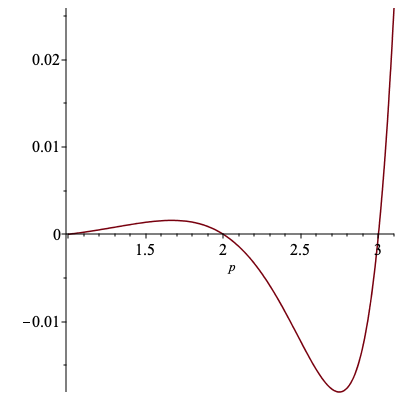}
	\caption{$||\sigma(A)+\sigma(B)||_p^p+||\sigma(A)-\sigma(B)||_p^p-||A+B||_p^p-||A-B||_p^p$ for $1\leq p\leq 3.1$, demonstrating the opposite expected behavior on the intervals $1\leq p\leq 2$ and $2\leq p\leq 3$.}
	\label{f1} 
\end{wrapfigure}

\s The obvious question is whether or not it is possible to relax the requirement that $A\geq B\geq 0$, perhaps even to $A+B, A-B\geq 0$.  The answer is: it is not.

\begin{counterexample}\label{ce1}
The matrices \begin{equation}
A=\begin{bmatrix} 6 & 0 \\ 0 & 5\end{bmatrix}, \qquad B=\begin{bmatrix} 0 & 1 \\ 1 & 0 \end{bmatrix}	
\end{equation}
have the property $A+B, A-B\geq 0$, and \begin{equation}
||A+B||_p^p+||A-B||_p^p \leq ||\sigma(A)+\sigma(B)||_p^p+||\sigma(A)-\sigma(B)||_p^p
\end{equation}
for $1\leq p\leq 2$ and $p\geq 3$, with the inequality reversing between $2\leq p\leq 3$.
\end{counterexample}

\s A plot of Counterexample \ref{ce1} can be seen in Figure \ref{f1}. This counterexample hinges on the fact that we chose $B$ to be unitary, so the ``up-down" rearrangement and the ``aligned" rearrangements were the same. In this case, as $A$ and $B$ satisfied the requirements of Theorem \ref{t2} (the extension of Conjecture \ref{c2}) but not of Theorem \ref{t1}, $||\sigma(A)\pm\sigma(B)||_p^p$ was treated as the ``up-down" and not the ``aligned" case.

\s Our proof of Theorem \ref{t2} is very similar to our proof of Theorem \ref{t1} which drew heavy inspiration from the proofs in \cite{Carlen2006}. However, it diverges from \cite{Carlen2006} in a very important manner: in \cite{Carlen2006}, the rearrangement inequalities in the integral representation required both $A,B\geq 0$. Therefore for Conjecture \ref{c2}, they first proved $||A+B||_p^p+||A-B||_p^p\leq ||A+|B|||_p^p+||A-|B|||_p^p$ for $1\leq p\leq 2$, then working with positive matrices $A$ and $|B|$ addressed the rearrangement. As monotonicity of $X^p$ was required, this does not extend as easily to  $2\leq p\leq 3$ as the proof of Theorem \ref{t1} did. We instead use majorization in the integral representation, removing the need to consider $|B|$ at all, which then allows us to extend the range without trouble:

\s

\s

\begin{theorem}\label{t2}
Let $A,B\in M_{n\times n}(\C)$ be self-adjoint with $A+B,A-B\geq 0$ and $\sigma_n(A)\geq \sigma_1(B)$. Then \begin{equation}\label{e1}
||A+B||_p^p+||A-B||_p^p\leq ||\sigma_\uparrow(A)+\sigma_{\downarrow}(B)||_p^p+||\sigma_\uparrow(A)-\sigma_{\downarrow}(B)||_p^p		
\end{equation}
with the inequality reversing for $2\leq p\leq 3$. There is equality for $p\neq 1,2$ if and only if $A$ and $B$ commute and they have simultaneous diagonalizations with diagonals $\sigma_\uparrow(A)$ and $\sigma_\downarrow(B)$, and hence there is equality in the entire range $1\leq p\leq 3$.
\end{theorem}
\begin{proof}
Once more, we use the integral representation. We can express the difference between sides in Equation \ref{e1} for $1<p<2$ by the integral representation after cancellation as \begin{equation}\label{e0}
k_p\Tr\left[\int_0^\infty\left(\frac{1}{A+B+t}+\frac{1}{A-B+t}-\frac{1}{\sigma_\uparrow(A)+\sigma_\downarrow(B)+t}-\frac{1}{\sigma_\uparrow(A)-\sigma_\downarrow(B)+t} \right)t^p dt\right]
\end{equation}

\s We will show that the integrand is always negative. We make the substitution $H=A+t$, $K=H^{-1/2}BH^{-1/2}$, then \begin{equation}
(A\pm B+t)^{-1}=H^{-1/2}(I\pm K)^{-1}H^{-1/2}=H^{-1/2}\left(\sum_{n=0}^{\infty}(-1)^n(\pm K)^n \right)H^{-1/2}
\end{equation}
and \begin{equation}\label{matside}
\frac{1}{A+B+t}+\frac{1}{A-B+t}=2H^{-1/2}\left(\sum_{n=0}^{\infty}K^{2n} \right)H^{-1/2}	
\end{equation}

\s For each $m$, we notice that $K^{2m}$ is a positive matrix, and hence $H^{-1/2}K^{2m}H^{-1/2}$ is positive, and hence the eigenvalues and singular values are the same. Therefore, \begin{align}
\Tr[H^{-1/2}K^{2m}H^{-1/2}]&=\sum_{i=1}^n\sigma_i(H^{-1/2}K^{2m}H^{-1/2})	\\
&=\sum_{i=1}^n\sigma_i(H^{-1/2}(H^{-1/2}BH^{-1/2})^{2m}H^{-1/2})	\\
&=\sum_{i=1}^n\sigma_i(H^{-1/2}(H^{-1/2}BH^{-1/2})^{m})^2	\\
&\leq\sum_{i=1}^n\sigma_i(H^{-1/2})^2\sigma_i((H^{-1/2}BH^{-1/2})^{m})^2	\\
&\leq\sum_{i=1}^n\sigma_i(H^{-1})\sigma_i(H^{-1/2})^{2m}\sigma_i(B)^{2m}\sigma_i(H^{-1/2})^{2m}	\\
&=\sum_{i=1}^n\sigma_{n+1-i}(H)^{-2m-1}\sigma_i(B)^{2n} \label{final}
\end{align}
This string makes repeated use of the majorization inequalities from Theorem \ref{sabweak} and Lemmas \ref{weakpower} and \ref{weaksum}. Furthermore, there is equality for $p\neq 1,2$ if and only if the integrand is always 0, and there is equality throughout. As we made use of log majorization $\sigma(AB)\prec_{(\log)}\sigma(A)\sigma(B)$, by Lemma \ref{strict} this must imply that $\sigma(AB)=\sigma(A)\sigma(B)$, which happens if and only if $A$ and $B$ commute with singular values aligned. Reversing the expansion trick from line \ref{final} gives $\frac{1}{\sigma_\uparrow(A)\pm\sigma_\downarrow(B)+t}$ as desired, completing our proof for $1<p<2$. The same integral representation for $2\leq p\leq 3$ as in Theorem \ref{t1} now extends the range.
\end{proof}

\begin{figure}[t!]
	\centering
	\begin{minipage}{0.47\textwidth}
		\centering
		\includegraphics[width=0.95\textwidth]{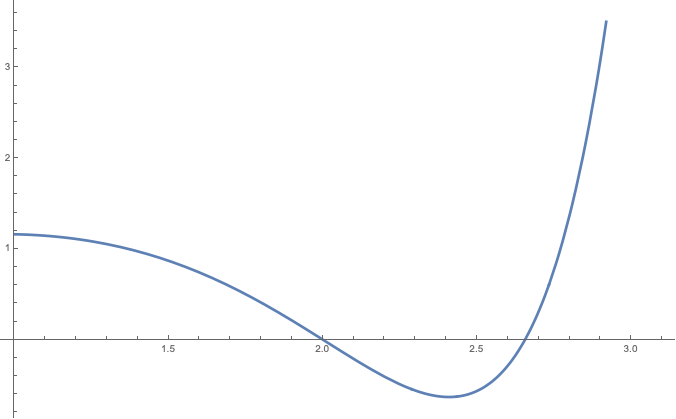} 
		\caption{$||\sigma_\uparrow(C)+\sigma_\downarrow(D)||_p^p+||\sigma_\uparrow(C)-\sigma_\downarrow(D)||_p^p-||C+D||_p^p-||C-D||_p^p$ for $1\leq p\leq 3$, demonstrating the expected behavior on the interval $1\leq p\leq 2$, and contrary behavior within $2\leq p\leq 3$.}
		\label{f2}
	\end{minipage}\hfill
	\begin{minipage}{0.47\textwidth}
		\centering
		\includegraphics[width=0.95\textwidth]{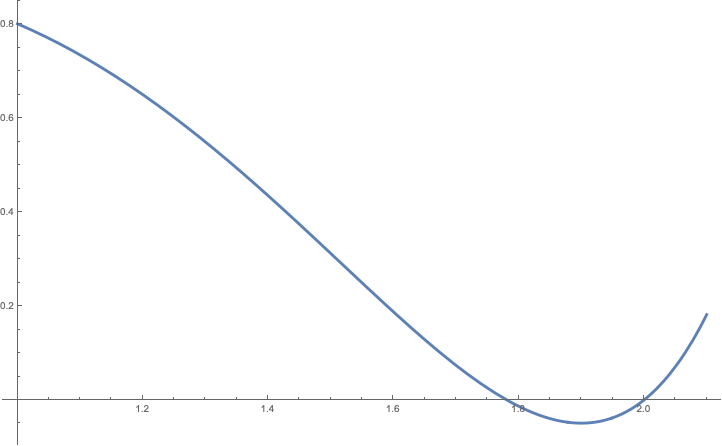} 
		\caption{$||\sigma(C)+\sigma(D)||_p^p+||\sigma(C)-\sigma(D)||_p^p-||C+D||_p^p-||C-D||_p^p$ with contrary behavior within $1\leq p\leq 2$.}
		\label{f3} 
	\end{minipage}
\end{figure}

\s An obvious counterexample to Conjecture \ref{c2} for all ranges are any pair of unitary matrices, as shown by Thoerem \ref{tu}. However, there are matrices that hold in the range $1\leq p\leq 2$, but not in the range $2\leq p\leq 3$, as demonstrated by Counterexample \ref{ce2} and Figure 2. In fact, these matrices $C$ and $D$ also provide a counterexample to Conjecture \ref{c1}, as seen in Figure 3.

\begin{counterexample}\label{ce2}
The matrices \begin{equation}
C=\begin{bmatrix} 6 & 0 \\ 0 & -1\end{bmatrix}, \qquad D=\begin{bmatrix} -1.97035 & 1.72243 \\ 1.72243 & 1.79035 \end{bmatrix}	
\end{equation}
are a counterexample for both Conjecture \ref{c1} and Conjecture \ref{c2}, with contrary behavior for Conjecture \ref{c1} within the interval $1\leq p\leq 2$; and contrary behavior for Conjecture \ref{c2} within the interval $2\leq p\leq 3$.
\end{counterexample}

\section*{Acknowledgements}
This research was funded by the NDSEG Fellowship, Class of 2017.  Thank you to my advisor, Professor Eric Carlen, for bringing my attention to the problem and providing me with a background to the subject.
\bibliographystyle{spmpsci}   
\bibliography{references} 
\end{document}